\documentclass[12pt]{article}
\usepackage{amsmath,amssymb,amsthm,graphicx}
\usepackage[british]{babel}
\usepackage{hyperref} 

\newtheorem{theorem}{Theorem}

\theoremstyle{definition}

\theoremstyle{remark}

 \allowdisplaybreaks

\newcommand{\1}[1]{{\mathbf{1}\mkern -1.5mu}{\{#1\}}}

\newcommand{\R}{{\mathbb R}}

\newcommand{\Z}{{\mathbb Z}}

\newcommand{\geo}[1]{\mathrm{Geom}_0 \left( {#1} \right)}

\newcommand{\IP}{{\mathbb P}}
\newcommand{\IE}{{\mathbb E}}

\title{On transience of $M/G/\infty$ queues}

\author{Serguei Popov\thanks{Centro de Matem\'atica, University of Porto, Porto, Portugal.
 E-mail:
 \texttt{serguei.popov@fc.up.pt}}}

\date{}

\begin{document}
\maketitle

%

\begin{abstract}
 We consider an $M/G/\infty$ queue
 with infinite 
expected service time. We then provide 
the transience/recurrence classification 
of the states (the system is said to be at state~$n$
if there are~$n$ customers being served), 
observing also that here (unlike
e.g.\ irreducible Markov chains) it is possible
for recurrent and transient states to coexist.
We also prove a lower bound on the growth
speed in the transient case.
\\[.3cm]\textbf{Keywords:} transience, recurrence,
service time, heavy tails
\\[.3cm]\textbf{AMS 2020 subject classifications:}
60K25, 60G55
\end{abstract}

In this note we consider a classical $M/G/\infty$
queue (see e.g.~\cite{Newell66}):
the customers arrive according to a Poisson
process with rate~$\lambda$; upon arrival, a customer
immediately enters to service, and the service times
are i.i.d.\ (nonnegative) 
random variables with some general distribution. 
For notational convenience, 
let~$S$ be a generic random variable 
with that distribution. 
We also assume that at time~$0$
there are no customers being served.
Let us denote by~$Y_t$ the number of customers
in the system at time~$t$, which we also refer to 
as the \emph{state of the system} at time~$t$;
note that, in general,
$Y$ is not a Markov process.

We are mainly interested in the situation where 
the system is \emph{unstable}, i.e., when $\IE S = \infty$.
In this situation, in principle,
our intuition tells us that the system 
can be \emph{transient} (in the sense $Y_t\to \infty$ a.s.)
or recurrent (i.e., all states are visited infinitely
often a.s.). 
However, it turns out that, for this model, the complete picture
is more complicated:
\begin{theorem}
\label{t_MGinfty_rec_trans}
Define
\begin{equation}
\label{df_k_0}
 k_0 = \min\Big\{k\in\Z_+ : 
 \int_0^\infty \big(\IE(S\wedge t)\big)^k
 \exp\big(-\lambda \IE(S\wedge t)\big)\, dt = \infty
  \Big\}
\end{equation}
(with the convention $\min\emptyset=+\infty$).
 Then
\begin{equation}
\label{eq_MGinfty_liminf}
\liminf_{t\to\infty}Y_t = k_0 \quad \text{a.s.}.
\end{equation}
In particular, if
\begin{equation}
\label{cond_MGinfty_trans}
 \int_0^\infty \big(\IE(S\wedge t)\big)^k
 \exp\big(-\lambda \IE(S\wedge t)\big)\, dt < \infty
\qquad \text{ for all }k\geq 0,
\end{equation}
then the system is transient;
if 
\begin{equation}
\label{cond_MGinfty_rec}
 \int_0^\infty 
 \exp\big(-\lambda \IE(S\wedge t)\big)\, dt = \infty,
\end{equation}
then 
the system is recurrent.
\end{theorem} 

Before proving this result, 
we make the following remark. 
Let us define~$M(t)$
to be the maximal remaining service time
of the customers which are present at time~$t$.
This is a so-called 
\emph{extremal shot noise process},
see~\cite{FouYua23} and references therein.
It is not difficult to obtain that transience
of~$M(\cdot)$ is the same as transience of
state~$0$ in $M/G/\infty$; then, Theorem~2.5
of~\cite{FouYua23} provides a criterion
for the transience of~$M(\cdot)$
(and therefore for the transience of state~$0$
in our situation). 

\begin{proof}[Proof of Theorem~\ref{t_MGinfty_rec_trans}.]
We start with a simple observation:
for any~$j\geq 0$, $\{\liminf Y_t =j\}$ is a tail event,
so it has probability~$0$ or~$1$. This implies 
that $\liminf Y_t$ is a.s.\ a constant 
(which may be equal to~$+\infty$).

 We use the following representation of the process
(see Figure~\ref{f_MGinfty}): consider a 
Poisson process in~$\R_+^2$, with the intensity 
measure $\lambda\, dt \times dF_S(u)$, 
where $F_S(u)=\IP[S\leq u]$ is the distribution function of~$S$.
Then, a point $(t,u)$ of this Poisson process is interpreted
in the following way: a customer arrived at time~$t$
and the duration of its service will be~$u$. Now,
draw a (dotted) line in the SE direction from each point,
as shown on the picture;
as long as this line stays in~$\R_+^2$, the corresponding
customer is present in the system. If we draw a vertical
line from~$(t,0)$ in the upwards direction,
then the number
of dotted lines it intersects is equal to~$Y_t$.
\begin{figure}
\begin{center}
\includegraphics{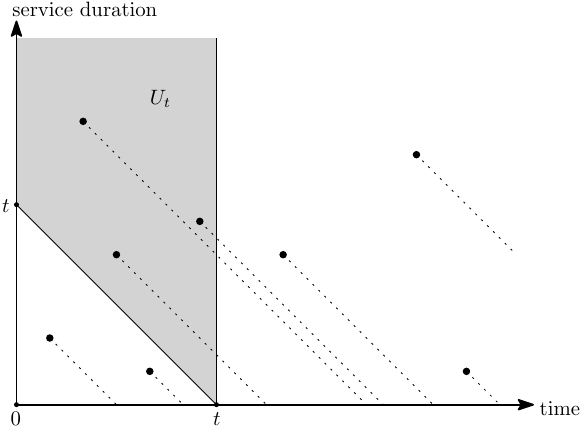}
\caption{A Poisson representation of $M/G/\infty$.
In this example, there are \emph{exactly} three customers
at time~$t$.}
\label{f_MGinfty}
\end{center}
\end{figure}

Next, for $k\in\Z_+$ denote by
\[
 T_k:=\{t\geq 0:Y_t=k\} 
\]
the set of time moments when the system has exactly~$k$
customers, and let 
\[
 U_{t} = \big\{(s,u)\in\R_+^2: s\in[0,t],
  u\geq t-s\big\}.
\]
We note that~$Y_t$ equals the number of points
in~$U_{t}$, which has Poisson distribution 
with mean
\[
 \int_{U_{t}}\lambda \, dt\, dF_S(u)
  = \lambda \IE(S\wedge t).
\]
Therefore, by Fubini's theorem, we have
(here, $|A|$ stands for the Lebesgue measure 
of~$A\subset \R$)
\begin{equation}
\label{E|T_k|}
\IE |T_k| = \IE \int_0^{\infty} \1{Y_t=k}\, dt
= \frac{\lambda^k}{k!}
\int_0^{\infty} \big(\IE(S\wedge t)\big)^k
\exp\big(-\lambda\IE(S\wedge t)\big)\, dt.
\end{equation}

Now, assume that $\IE|T_k|<\infty$ for some~$k\geq 0$;
it automatically implies that $\IE|T_\ell|<\infty$
for $0\leq\ell\leq k$. 
This means that $|T_0|,\ldots,|T_k|$
are a.s.\ finite, and let us show that $T_0,\ldots,T_k$
have to be a.s.\ bounded (this is a small technical
issue that we have to resolve because we are considering
continuous time). Probably, the cleanest way to see this
is the following:
first, notice that, in fact, $T_0$ is a union of intervals 
of random i.i.d.\ (with Exp($\lambda$)
distribution) lengths,
because each time when the system becomes empty, it 
will remain so till the arrival of the next customer.
Therefore,
$|T_0|<\infty$ clearly means that~$\sup T_0\leq K_0$ 
for some (random)~$K_0$. Now, \emph{after}~$K_0$ 
there are no~$1\to 0$ transitions anymore, so
the remaining part of~$T_1$ again
becomes a union of such intervals, meaning that
it should be bounded as well; we then repeat this 
reasoning a suitable number of times to finally
obtain that~$T_k$ must be a.s.\ bounded. 
This implies that 
$\liminf_{t\to \infty} Y_t \geq k_0$ a.s..

Next, assume that $\{0,\ldots,k\}$ is 
a \emph{transient set},
in the sense that $\liminf_{t\to \infty} Y_t \geq k+1$ a.s.; let us show that this implies 
that $\IE|T_k|<\infty$.
Indeed, first,
we can choose a sufficiently large~$h>0$
in such a way that 
\[
 \IP\big[Y_t\geq k+1 \text{ for all }t\geq h\big] 
\geq \frac{1}{2}.
\]
Define a stopping time $\tau=\inf\{t\geq h : Y_t\leq k\}$ 
(again, with the convention $\inf\emptyset = +\infty$).
Then, a crucial observation is that what one sees
after~$\tau$ is a superposition of two \emph{independent}
systems: one is formed by those customers
(with their remaining lifetimes) present at~$\tau$,
and the other is a copy of the original system.
Then, a simple coin-tossing argument
together
with the fact that an initially nonempty system
(i.e., with 
some customers being served, with any assumptions on
their remaining service times) 
dominates an initially
empty system
show that~$|T_k|$ (in fact, $|T_0|+\cdots+|T_k|$)
is dominated by $h\times\geo{\frac{1}{2}}$
random variable and therefore
has a finite expectation. It means that we have
$\liminf_{t\to \infty} Y_t \leq k_0$ 
a.s.\ (because otherwise, 
in the situation when $k_0<\infty$,
we would have $\IE|T_{k_0}|<\infty$,
which, by definition, is not the case).
This concludes the proof of
Theorem~\ref{t_MGinfty_rec_trans}.
\end{proof}

Regarding this result, we may observe that,
in most situations one would have~$k_0=0$ or~$+\infty$;
this is because convergence of such integrals
is usually determined by what is in the exponent.
Still, it is not difficult to construct ``strange examples''
with $0<k_0<\infty$, i.e., where the process will
visit $\{0,\ldots, k_0-1\}$ only finitely many times,
but will hit every~$k\geq k_0$ infinitely often 
a.s.\ (a behaviour one cannot have with 
irreducible Markov chains).
For instance, let $\lambda=1$ and fix $b>0$;
next,
consider a service time distribution such that
 $1-F_S(u)=\frac{1}{u}+\frac{b}{u\ln u}$ 
 for large enough~$u$.
Then it is elementary to obtain
that $\IE(S\wedge t) = \ln t + b \ln\ln t + O(1)$
and the integrals in~\eqref{df_k_0} diverge
whenever $k\geq b-1$, 
meaning that~$k_0 = \lceil b \rceil-1$.

Now, in the situation
when~\eqref{cond_MGinfty_trans} holds and~$Y$
is transient, it may also be useful
to be able to say something about the speed
of convergence of~$Y_t$ to infinity.
We do not 
intend to enter deeply into this question here,
but only prove a particular
result needed for future reference.
Namely, in~\cite{MPW24} we work with a different
model which in some sense dominates $M/G/\infty$;
so, we will now give a lower bound on the growth
of~$Y_t$,
more specifically, we will show that under certain
conditions~$Y_t$ will be eventually at least
a constant fraction of its expected value.
For $q\in(0,1)$, let us define 
$\gamma_q = 1 - q - q\ln q^{-1} > 0$.

\begin{theorem}
\label{t_growth_Y} 
Fix $q\in(0,1)$ and assume that
\begin{equation}
\label{cond_growth_Y}
\int_0^\infty \exp\big(-\gamma_q\lambda \IE(S\wedge t)\big)\, dt < \infty.
\end{equation}
Then 
\begin{equation}
\label{eq_growth_Y}
\IP\big[Y_t\geq q\lambda \IE(S\wedge t)
\text{ for all large enough }t\big]=1.
\end{equation}
\end{theorem}

\begin{proof}
 Let 
\[
 H_q = \big\{t\geq 0: 
 Y_t < q\lambda \IE(S\wedge t) \big\};
\]
our goal is to show that~$H_q$ is a.s.\ bounded
in the case when~\eqref{cond_growth_Y} holds.
 We recall a standard (Chernoff) tail bound: 
if~$X$ is Poisson($\mu$) and $q\in(0,1)$, then
\begin{equation}
\label{Chernoff_Poisson}
 \IP[X\leq q\mu] 
\leq \exp\big(-(q\mu\ln q + \mu-q\mu)\big)
= \exp(-\gamma_q \mu).
\end{equation}
Then, analogously to~\eqref{E|T_k|}
we obtain from~\eqref{Chernoff_Poisson} that
\begin{equation}
\label{est_E|H_q|}
 \IE |H_q| \leq 
 \int_0^\infty \exp\big(-\gamma_q\lambda \IE(S\wedge t)\big)\, dt;
\end{equation}
so, by~\eqref{cond_growth_Y},
we have $\IE|H_q|<\infty$, meaning that $|H_q|<\infty$
a.s..
To see that this has to imply that~$H_q$
is a.s.\ bounded, 
analogously to the proof of
Theorem~\ref{t_MGinfty_rec_trans}
one can reason in the following
way. If $t\in H_q$, then $s\in H_q$ for all 
$s\in (t, A_t)$, where~$A_t$ is the first moment
after~$t$ when a customer arrives to the system.
This implies that the lengths of the intervals
that constitute~$H_q$ dominate a sequence
of i.i.d.\ random variables with Exp($\lambda$)
distribution; by its turn, this clearly
implies that if~$|H_q|$ is finite then it has
to be bounded.
\end{proof}

\end{document}